\documentclass{amsart}
\usepackage[utf8]{inputenc}
\usepackage[utf8]{inputenc}
\usepackage{amsthm}
\theoremstyle{definition} 
\theoremstyle{plain} \newtheorem{lemma}{Lemma}[section]
\theoremstyle{plain} 
\theoremstyle{plain} \newtheorem{theorem}{Theorem}[section]
\theoremstyle{plain} \newtheorem{corollary}{Corollary}[section]
\theoremstyle{plain} 
\theoremstyle{plain} 
\theoremstyle{plain} 
\theoremstyle{remark} 
\theoremstyle{remark} 
\theoremstyle{plain} 
\theoremstyle{remark} \newtheorem{remark}{Remark}[section]
\usepackage{verbatim}
\usepackage{textcomp}
\usepackage{latexsym}
\usepackage{amsfonts}
\usepackage{amssymb}
\usepackage{amsmath}
\usepackage{mathrsfs}
\usepackage{systeme}
\numberwithin{equation}{section}

\newcommand{\namedthm}[2]{\theoremstyle{plain}
   \newtheorem*{thm#1}{#1}\begin{thm#1}#2\end{thm#1}}

\def\bcw{\mathbin{\bigcirc\mkern-15mu\wedge}}
\title{Rigidity for Bach-flat metrics on manifolds with boundary and applications}

\author{Matthew J. Gursky}
\address{Department of Mathematics \\
         University of Notre Dame\\
         Notre Dame, IN 46556}

\author{Siyi Zhang}
\address{Department of Mathematics \\
         University of Notre Dame\\
         Notre Dame, IN 46556}

\date{\today}
\begin{document}
\maketitle
\section{Statement of main result}

In this paper we study a conformally invariant boundary value problem in four dimensions.  Our work is partially inspired by the following rigidity result of Hang-Wang in \cite{HW06}:

\namedthm{Theorem A}{(\cite{HW06}, Theorem 4.1) Let $(M,g)$ be a smooth $n$-dimensional compact Einstein
manifold with boundary $\Sigma$. If $\Sigma$ is totally geodesic and is isometric to $S^{n-1}$
with the standard metric, then $(M,g)$ is isometric to the hemisphere $S^n_{+}$
with the standard metric.
}

 This result can be viewed as a uniqueness statement for solutions of an overdetermined boundary value problem for Einstein metrics.  The assumption that the induced metric is round plays the role of the Dirichlet data, while the assumption that the boundary is totally geodesic is the Neumann data.   Theorem A states that the unique solution of the Einstein equation in $M^n$ satisfying both of these boundary conditions is the upper hemisphere with the standard metric.  For classical elliptic PDE the model for such a uniqueness result is the famous symmetry theorem of Serrin \cite{Serrin}.

There is also a variational interpretation of Theorem A.  Given a Riemannian metric $g$ defined on the manifold with boundary $(M,\Sigma)$, let $R_g$ denote the scalar curvature of $g$ and $H_g$ the mean curvature (i.e., the trace of the second fundamental form) of the boundary.  Let $\mathfrak{M}(M)_1$ denote the space of unit volume metrics on $M$.  In \cite{Araujo}, Araujo showed that critical points of the functional
\begin{align} \label{EH}
\mathcal{E}_b : g \mapsto \int_M R_g \, dv_g + 2 \int_{\Sigma} H_g \, d\sigma_g
\end{align}
restricted to $\mathfrak{M}_1$ correspond to Einstein metrics with totally geodesic boundary.  Therefore, we can restate Theorem A in the following way:

\namedthm{Theorem B}{The upper hemisphere $S^{n}_{+}$ with the standard metric is the unique critical point (up to isometry) of $\mathcal{E}_b \big\vert_{\mathfrak{M}_1}$ such that the induced metric on $\Sigma$ is isometric to the round sphere.
}

In this paper we consider a high order version of Theorems A and B in four dimensions.  In order to state our results we will need some additional notation. 
 
From now on, we assume $(M^4,\Sigma^3 = \partial M^4,g)$ is a compact four-dimensional Riemannian manifold with boundary.  Let $W_g$ denote the Weyl curvature tensor of $g$ and $L$ the second fundamental form of the boundary.   In place of $\mathcal{E}_b$, consider the functional
\begin{align} \label{W}
\mathcal{W}_b : g \mapsto \int_{M^4}||W_{g}||^2\, dv_g + 2\oint_{\Sigma^3}W_{i0j0}L^{ij}\, d{\sigma}_g,
\end{align}
where $0$ subscripts correspond to components of a tensor with respect to the outward unit normal, and $\| \cdot \|$ is the norm of $W$ as a section of $End(\Lambda^2(M))$.  This functional generalizes the Weyl functional
\begin{align*}
\mathcal{W} : g \mapsto \int_{M^4} \| W_g \|^2 \, dv_g
\end{align*}
for closed manifolds.  Critical points of $\mathcal{W}$ are metrics with vanishing Bach tensor $B_{\alpha \beta}$ defined by 
\begin{align} \label{BF} 
B_{\alpha\beta} = \nabla^{\gamma}\nabla^{\delta} W_{\alpha\gamma\beta\delta} + P^{\gamma\delta}W_{\alpha\gamma\beta\delta},
\end{align}
where $P$ is the Schouten tensor (see Section \ref{prelim} for more details).  Four-manifolds with vanishing Bach tensor are also called {\em Bach-flat} manifolds.  We remark that $\mathcal{W}$ and $\mathcal{W}_b$ are conformally invariant, hence Bach-flatness and $S$-flatness are conformally invariant conditions.

As pointed out in \cite{CG18}, critical points of $\mathcal{W}_b$ are Bach-flat metrics such that the tensor
\begin{align} \label{S}
S_{ij} := \nabla^{\alpha}W_{{\alpha}i0j} + \nabla^{\alpha}W_{{\alpha}j0i} - \nabla^0W_{0i0j} + \frac{4}{3}HW_{0i0j}
\end{align}
vanishes on the boundary.  In this case, we will say that the boundary is $S$-flat.  Since the Bach-flat condition is fourth order in the metric, it should be possible to specify a boundary condition in addition to $S$-flatness.  In the Appendix, we prove the following:

\begin{theorem} \label{AppThm}  Given a compact four-dimensional manifold with boundary $(M^4, \Sigma^3 = \partial M^4)$, let $\mathcal{M}^{0}(M^4,\Sigma^3)$ denote the space of all Riemannian metrics on $(M^4,\Sigma^3)$ such that $\Sigma^3$ is umbilic.  Then $g$ is a critical point of 
\begin{align*}
\mathcal{W} \big|_{\mathcal{M}^{0}(M^4,\Sigma^3)},
\end{align*}
if and only if $g$ is Bach-flat and $\Sigma^3$ is $S$-flat and umbilic. 
\end{theorem}

We remark that when the boundary is umbilic the functionals $\mathcal{W}$ and $\mathcal{W}_b$ are actually the same, since $W_{i0j0}L^{ij} \equiv 0$.  

Our goal is to prove a uniqueness result for critical points of the variational problem described in Theorem \ref{AppThm}.  Due to conformal invariance of the functional and the constraint any uniqueness result can only hold modulo conformal changes of metric, unless a choice of conformal representative is specified.
A natural candidate for a conformal representative is a Yamabe metric.  

Given a compact manifold with boundary $(M^4,\Sigma^3,g)$, let $[g] = \{ e^{2f} g\, : f \in C^{\infty}(M) \}$ denote the conformal class of $M$.  If we restrict the functional $\mathcal{E}_b$ in (\ref{EH}) to unit-volume metrics in $[ g ]$, then critical points are precisely those metrics with constant scalar curvature and zero mean curvature on the boundary.  The {\em first Yamabe invariant} of $(M^4,\Sigma^3,g)$ is the infimum of $\mathcal{E}_b$ (restricted to unit volume metrics):
\begin{align} \label{Y1}
    \mathcal{Y}(M^4,\Sigma^3,[g]) = \inf_{\widetilde{g}\in[g],  Vol(\tilde{g}) = 1} \left(\int_M R_{\widetilde{g}}\,\, dv_{\widetilde{g}}+2\int_{\Sigma}H_{\widetilde{g}}\,\, d\sigma_{\widetilde{g}}\right)
\end{align}
By the work of Escobar \cite{Esc92}, there is always a metric $g_Y \in [g]$ that attains $\mathcal{Y}(M^4,\Sigma^3,[g])$ (see Section \ref{prelim} for more details).   However, $g_Y$ need not be unique: the round metric $g_0$ on $S^4_{+}$ is Yamabe, but for any conformal transformation $\varphi : (S^4_{+},S^3,g_0) \rightarrow (S^4_{+},S^3,g_0)$, the metric $\widetilde{g}_0 = \phi^{\ast}g_0$ is also Yamabe.  

With these preliminaries, we can now state our main result:

\begin{theorem}\label{main}
Let $(M^4,\Sigma^3,g)$ be a Bach-flat Riemannian four-manifold with boundary such that the boundary is $S$-flat and umbilic.   Suppose for some Yamabe metric $g_Y \in [g]$ with $R_{g_Y} = 12$, the induced metric $g_Y \vert_{\Sigma^3}$ is isometric to $S^3$ with the standard metric.   Then $(M^4, \Sigma^3, g_Y)$ is isometric to the hemisphere $S^4_{+}$ with the standard metric.  
\end{theorem}

In view of Theorem \ref{AppThm}, we have the following corollary (compare with Theorem B):

\begin{corollary}  \label{MainCor} The upper hemisphere $S^4_{+}$ with the standard metric is the unique critical point of $\mathcal{W} \big|_{\mathcal{M}^{0}(M^4,\Sigma^3)}$ admitting a Yamabe metric $g_Y$ such that the induced metric on the boundary is isometric to the round $S^3$.
\end{corollary} 

As with the Hang-Yang result, Theorem \ref{main} can be viewed as a uniqueness result (modulo conformal transformations) for an overdetermined boundary value problem.   In this case, the Bach-flat condition is fourth order in the metric, so it is natural to impose two boundary conditions on the metric $g$; i.e., $S$-flatness and umbilic, which are third order and fist order respectively in the metric.  The additional assumption on the Yamabe metric is  a kind of ``conformally invariant Dirichlet condition'', and makes the problem overdetermined.

At first glance it may seem that the assumption on the Yamabe metric is too strong, and it would be more natural to just assume that the metric $g$ when restricted to the boundary is conformal to the round $S^3$.  However, by the work of Schoen-Yau \cite{SchoenYau} one can construct exaples of manifolds satisfying this weaker condition that are not even diffeomorphic to the upper hemisphere:

\begin{theorem}  \label{SYthm}  (See \cite{SchoenYau}) The manifold with boundary $(S^3 \times S^1 \setminus B^4, S^3)$, where $B^4$ is a four-dimensional ball, admits a metric $\widetilde{g}$ with the following properties:   \begin{enumerate}  \smallskip

\item $\widetilde{g}$ is locally conformally flat, hence Bach-flat and $S$-flat;   \smallskip

\item The boundary $S^3$ is umbilic with respect to $\widetilde{g}$;  \smallskip

\item The induced metric $\widetilde{g} \vert_{S^3}$ is conformal to the round metric $h_0$ on $S^3$.

\end{enumerate}

 \end{theorem}

In fact, for any $k \geq 1$ the Schoen-Yau construction implies the existence of a metric $\widetilde{g}_k$ on $ k \sharp (S^3 \times S^1)$ with the same properties.  In view of Theorem \ref{SYthm}, one needs a stronger condition on the induced metric in order to distinguish the upper hemisphere among Bach-flat and $S$-flat manifolds with umbilic boundary.

To conclude the introduction we point out that four-dimensional Bach-flat manifolds with umbilic boundary arise naturally in the context of theory of conformally compact Einstein (CCE) manifolds.  CCE manifolds are central to the Fefferman-Graham theory of conformal invariants, and appear in the physics literature in the AdS/CFT correspondence.  Here, we give a very brief explanation of the connection to our work, and refer the reader to \cite{FG} for more details.

Suppose $X$ is the interior of a smooth, compact manifold with boundary $( \bar{X}, N = \partial X)$.  A metric $g_{+}$ defined in $X$ is {\em conformally compact} if there is a defining function for the boundary $\rho : \bar{X} \rightarrow \mathbb{R}$ such that $\bar{g} = \rho^2 g_{+}$ defines a metric on $\bar{X}$.  By a defining function, we mean a smooth function with $\rho > 0$ in $X$, $\rho = 0$ and $d \rho \neq 0$ on $\partial X$.  We will assume in the following that $\bar{g}$ is at least $C^2$ up to the boundary.  If $(X, \partial X, g_{+})$ is Einstein, then we say that $(X,\partial X, g_{+})$ is a conformally compact Einstein (CCE) manifold.

The choice of defining function is not unique, and thus a conformally compact manifold $(X, N = \partial X , g_{+})$ naturally defines a conformal class of metrics on the boundary, $[h]$, called the {\em conformal infinity}.  Given a metric $h$ in the conformal infinity there is a canonical choice of defining function, called a {\em special} or {\em geodesic} defining function $r$, such that near the boundary $\bar{g} = r^2 g_{+}$ can be written as
\begin{align} \label{FG1} 
    \bar{g} = dr^2 + h_r
\end{align} 
where $h_r$ is a one-parameter family of metrics on $N$.  Moreover, the boundary $N$ is totally geodesic with respect to $\bar{g}$.   

Now suppose $(X^4, N^3 = \partial X^4, g_{+})$ is a four-dimensional CCE manifold.  Given $h$ in the conformal infinity, let $\bar{g} = r^2 g_{+}$ be the compactification by the special defining function assoicated to $h$.  Since $g_{+}$ is Einstein, it is Bach-flat.  By conformal invariance of the Bach-flat condition $\bar{g}$ is also Bach-flat.  As we observed above,  $N^3$ is totally geodesic (hence umbilic) with respect to $\bar{g}$.  Moreover, the metric $h_r$ in (\ref{FG1}) can be expanded near $N^3$ to give
\begin{align} \label{FGE}
    \bar{g} = dr^2 + h + g^{(2)}r^2 + g^{(3)}r^3 + O(r^4),
\end{align}
where $g^{(2)}$ and $g^{(3)}$ are tensors on $N^3$.  As shown in \cite{FG}, $g^{(2)}$ is determined by the metric $h$, but $g^{(3)}$ is formally undetermined.  In \cite{CG18}, Chang-Ge showed that
\begin{align*}
    S_{\bar{g}}= -\frac{3}{2}g^{(3)}.
\end{align*}
To summarize:  Four-dimensional CCE manifolds provide many examples of Bach-flat manifolds with umbilic boundary.  Moreover, the vanishing of the $S$-tensor has a concrete interpretation via the Fefferman-Graham expansion (\ref{FGE}).  We remark that the vanishing of $S$ can be used in some cases to characterize the geometry; see \cite{LQS17}.

\bigskip

\section{Preliminaries} \label{prelim}

\subsection{Basic notations and properties for manifolds with boundary}
Suppose $(M^n,\Sigma^{n-1}, g)$ is a Riemannian manifold with boundary $(\Sigma^{n-1},h)$, where $h = g|_{\Sigma}$ is the induced metric.
Throughout this note, we denote the Riemannian curvature tensor by $Rm$ (or $Rm_g$ if we need to specify the metric), the Ricci tensor by $Ric$, and the scalar curvature by $R$.  We also denote the Weyl curvature tensor by $W$, and the Schouten tensor
\begin{align} \label{Pdef}
P = \frac{1}{n-2} \left( Ric - \frac{1}{2(n-1)}R \cdot g \right).
\end{align}
In terms of the Weyl and Schouten tensors the Riemannian curvature tensor can be decomposed as
\begin{align} \label{rot}
Rm = W + P \bcw g
\end{align}
where $\bcw$ is the Kulkarni-Nomizu product. We use $Rm^{\Sigma}$, $W^{\Sigma}$, $Ric^{\Sigma}$, $P^{\Sigma}$, and $R^{\Sigma}$ to denote the respective curvature tensors calculated with respect to the intrinsic metric $h$ on $\Sigma^{n-1}$.

The boundary is called \emph{umbilic} if
\begin{align}
    L_{ij} = \lambda{h_{ij}},
\end{align}
where $\lambda$ is a smooth function on $\Sigma^{n-1}$ and $L_{ij}$ is the second fundamental form of $\Sigma^{n-1}$. In other words, the boundary is umbilic if its second fundamental form is pointwise proportional to the metric. By taking trace, we obtain that $\lambda = \frac{H}{n-1}$, where $H$ is the mean curvature of $\Sigma^{n-1}$. The boundary is called \emph{minimal} if its mean curvature is vanishing, i.e., $H = 0$. The boundary is called \emph{totally geodesic} if its second fundamental form is vanishing, which is equivalent to the fact that the boundary is minimal and umbilic.
Note that the umbilic condition is conformally invariant: if $(\Sigma^{n-1},h)$ is umbilic with respect to the metric $g$ and $\widetilde{g} = u^2g$ is a metric conformal to $g$, then $(\Sigma^{n-1},\widetilde{h})$ is also umbilic with respect to the metric $\widetilde{g}$.

The \emph{first Yamabe invariant} of $(M^n,\Sigma^{n-1}, g)$ is defined as
\begin{align}
    \mathcal{Y}(M^n,\Sigma^{n-1},[g]) = \inf_{\widetilde{g}\in[g]}Vol(\tilde{g})^{-\tfrac{n-2}{n}} \left(\int_MR_{\widetilde{g}}\,\, dv_{\widetilde{g}}+2\int_{\Sigma}H_{\widetilde{g}}\,\, d\sigma_{\widetilde{g}}\right)
\end{align}
Any smooth metric achieving this infimum has constant scalar curvature and minimal boundary. From the work of Escobar \cite{Esc92}, it is known that in many cases such a minimizer exists. In particular, for $3\leq{n}\leq{5}$, a minimizer always exists.  In this note, we shall call the minimizing metric scaled to have constant scalar curvature $n(n-1)$ and minimal boundary a \emph{Yamabe metric} in its conformal class. In addition, Escobar established the following inequality for $3\leq{n}\leq{5}$:
\begin{align}
    \mathcal{Y}(M^n,\Sigma^{n-1},[g])\leq\mathcal{Y}(S^n_+,S^{n-1},[g_{S^n_{+}}]),
\end{align}
where equality holds if and only if $(M^n,\Sigma^{n-1},g)$ is conformally equivalent to the round upper hemisphere $(S_+^n,S^{n-1},g_{S^n_{+}})$. Note that for a manifold with umbilic boundary, the Yamabe metric has constant scalar curvature and totally geodesic boundary.

\subsection{The Weyl functional on four-manifolds with boundary}
On a closed smooth four-manifold, the Weyl functional is defined as
\begin{align}
       \mathcal{W}:\,\,\, g \,\,\, \to \,\,\, \int_{M^4}||W_{g}||^2\, dv_g.
    \end{align}
It has played an important role in the study of the geometry and topology of the underlying manifold. On a smooth four-manifold with boundary$(M^4, \Sigma^{3})$, the Weyl functional is defined as
\begin{align}
       \mathcal{W}_b:\,\,\, g \,\,\, \to \,\,\, \int_{M^4}||W_{g}||^2\, dv_g + 2\oint_{\Sigma^3}W_{i0j0}L^{ij}\, d{\sigma}_g,
\end{align}
where $L_{ij}$ and $H$ are the second fundamental form and mean curvature of $\Sigma^3$, respectively, Latin letters run through $1,2,3$ as tangential directions, and $0$ is the outward normal direction on $\Sigma$. The functional $\mathcal{W}_b$ is conformally invariant in four dimensions in the sense that $\mathcal{W}_b(\widetilde{g}) = \mathcal{W}_b(g)$ for any $\widetilde{g}\in[g]$. Indeed, $||W_{g}||^2\, dv_g$ and $W_{i0j0}L^{ij}\, d{\sigma}_g$ are pointwise conformally invariant differential forms in $M^4$ and on $\Sigma^3$, respectively. Also note that for umbilic boundary, $W_{i0j0}L^{ij} \equiv 0$ on $\Sigma^3$ since Weyl curvature is trace-free. It follows that $\mathcal{W}_b$ coincides with $\mathcal{W}$ on four-manifolds with umbilic boundary.

As mentioned in the Introduction (a proof will be given in the Appendix), critical points of $\mathcal{W}_{b}$ are Bach-flat metrics in $M^4$ with vanishing $S$-tensor on $\Sigma^3$. The basic conformal properties of the Bach tensor and the $S$-tensor are given in the following lemma:

\begin{lemma}[\cite{CGY02}\cite{CG18}\cite{Der83}]\label{S tensor}
The Bach tensor $B_{\alpha\beta}$ and $S$-tensor $S_{ij}$ on $(M^4,\Sigma^3,g)$ have the following properties:
\begin{enumerate}
    \item $B_{\alpha\beta}$ is symmetric, trace-free, divergence-free and conformally invariant in the sense that for $\widetilde{g} = e^{2w}g$,
           \[B_{\widetilde{g}} = e^{-2w}B_g.\]
    \item $S_{ij}$ is symmetric, trace-free and conformally invariant in the sense that for $\widetilde{g} = e^{2w}g$,
           \[S_{\widetilde{g}} = e^{-w}S_g.\]
    \item If $\Sigma^3$ is totally geodesic, then
    \[S_{ij} = \nabla^{0}P_{ij}.\]
\end{enumerate}
\end{lemma}

\section{Weyl curvature on umbilic boundary}
In this section, we list and prove several useful properties of the Weyl curvature tensor on umbilic boundary.  

\begin{lemma}\label{Weyl umbilic boundary}
Suppose $(M^4,\Sigma^3,g)$ has umbilic boundary. Then on $\Sigma^3$
\begin{align}
     W_{0i0j} & = P_{ij} - P^{\Sigma}_{ij} + \frac{1}{18}H^2g_{ij}, \\
     W_{ijk0} & = 0, \\
     \sum_{i,j,k,l=1}^{3}|W_{ijkl}|^2 & = 4\sum_{i,j = 1}^3|P_{ij} - P^{\Sigma}_{ij} + \frac{1}{18}H^2g_{ij}|^2.
\end{align}
In particular, $W = 0$ on $\Sigma^3$ if and only if $W_{0i0j} = 0$ on $\Sigma^3$.
\end{lemma}
\begin{proof}
Recall that $(\Sigma^3,h)$ being umbilic means that
\begin{align}\label{umbilic}
    L_{ij} = \frac{1}{3}Hg_{ij}.
\end{align}

With (\ref{umbilic}), the Gauss equations imply on $\Sigma^3$ that
\begin{align}\label{Gauss}
    R_{ikjl} = R^{\Sigma}_{ikjl} - L_{ij}L_{kl} + L_{il}L_{jk} = R^{\Sigma}_{ikjl} - \frac{1}{9}H^2g_{ij}g_{kl} + \frac{1}{9}H^2g_{il}g_{jk}
\end{align}
Taking the trace, we have on $\Sigma^3$ that 
\begin{align}\label{Gauss 1st trace}
    R_{ij} - R_{0i0j} = R^{\Sigma}_{ij}  - \frac{2}{9}H^2g_{ij}.
\end{align}
Taking the trace once more, we have on $\Sigma^3$ that 
\begin{align}\label{Gauss 2nd trace}
R - 2R_{00} = R^{\Sigma} -\frac{2}{3}H^2
\end{align}
The decomposition of curvature implies on $\Sigma^3$ that
\begin{align}\label{R_ijkl on boundary}
R_{0i0j} = W_{0i0j} + g_{00}P_{ij} - g_{ij}P_{00}.
\end{align}
By the definition of Schouten tensor, we have
\begin{align}\label{Schouten}
P_{00} = \frac{1}{2}\left(R_{00} - \frac{1}{6}Rg_{00}\right),\,\,\, P_{ij} = \frac{1}{2}\left(R_{ij} - \frac{1}{6}Rg_{ij}\right),\,\,\, P^{\Sigma}_{ij} = R^{\Sigma}_{ij} - \frac{1}{4}R^{\Sigma}_{ij}.
\end{align}
If we substitute (\ref{R_ijkl on boundary}) and (\ref{Schouten}) into (\ref{Gauss 1st trace}), then 
\begin{align}
    2P_{ij} +\frac{1}{6}Rg_{ij} - W_{0i0j} - P_{ij} - P_{00}g_{ij} = P^{\Sigma}_{ij} +\frac{1}{4}R^{\Sigma}g_{ij} -\frac{2}{9}H^2g_{ij},
\end{align}
which implies
\begin{align}\label{Weyl on the boundary}
    W_{0i0j} = P_{ij} - P^{\Sigma}_{ij} + \left(\frac{1}{6}R-\frac{1}{4}R^{\Sigma}-P_{00} + \frac{2}{9}H^2\right)g_{ij}.
\end{align}
Also, substuting (\ref{Schouten}) into (\ref{Gauss 2nd trace}) gives: 
\begin{align}\label{Gauss 2nd trace Schouten}
    P_{00} = \frac{1}{6}R - \frac{1}{4}R^{\Sigma} +\frac{1}{6}H^2.
\end{align}
Finally, substituting (\ref{Gauss 2nd trace Schouten}) into (\ref{Weyl on the boundary}) we get 
\begin{align}
    W_{0i0j} = P_{ij} - P^{\Sigma}_{ij} + \frac{1}{18}H^2g_{ij}.
\end{align}

By (\ref{umbilic}), the Codazzi equations imply
\begin{align}\label{Codazzi}
    R_{ijk0} = -\nabla^{\Sigma}_{j}L_{ik} + \nabla^{\Sigma}_{i}L_{jk} = -\frac{1}{3}\nabla^{\Sigma}_{j}Hg_{ik} + \frac{1}{3}\nabla^{\Sigma}_{i}Hg_{jk}
\end{align}
Taking the trace, we have on $\Sigma^3$ that
\begin{align}\label{R_j0 on the boundary}
    R_{j0} = -\frac{2}{3}\nabla^{\Sigma}_{j}H.
\end{align}
The decomposition of curvature implies on $\Sigma^3$ that
\begin{align}\label{R_ijk0 on boundary}
R_{ijk0} = W_{ijk0} + g_{ik}P_{j0} - g_{jk}P_{i0}
\end{align}
By definition, we have from (\ref{R_j0 on the boundary}) on $\Sigma^3$ that
\begin{align}\label{P_i0}
    P_{i0} = \frac{1}{2}\left(R_{i0} - \frac{1}{6}Rg_{i0}\right) = \frac{1}{2}R_{i0} = -\frac{1}{3}\nabla^{\Sigma}_{i}H.
\end{align}
Combining  (\ref{Codazzi}),(\ref{R_ijk0 on boundary}), and (\ref{P_i0}), we have on $\Sigma^3$
\begin{align}
    -\frac{1}{3}\nabla^{\Sigma}_{j}Hg_{ik} + \frac{1}{3}\nabla^{\Sigma}_{i}Hg_{jk} = W_{ijk0} - \frac{1}{3}\nabla^{\Sigma}_{j}Hg_{ik} + \frac{1}{3}\nabla^{\Sigma}_{i}Hg_{jk},
\end{align}
which implies $W_{ijk0} = 0$ on $\Sigma^3$.

Next, we write the Gauss equation (\ref{Gauss}) using the decomposition of $Rm$ into $W$, $P$ and $R$. Recall
\begin{align}\label{Rie decomp M}
    R_{ikjl} = W_{ikjl} + g_{ij}P_{kl} + g_{kl}P_{ij} - g_{il}P_{kj} - g_{kj}P_{il},
\end{align}
and similarly 
\begin{align}\label{RIe decomp Sigma}
    R^{\Sigma}_{ikjl} = h_{ij}P^{\Sigma}_{kl} + h_{kl}P^{\Sigma}_{ij} - h_{il}P^{\Sigma}_{kj} - h_{kj}P^{\Sigma}_{il},
\end{align}
where we have used $W^{\Sigma}_{ijkl} = 0$ since the Weyl curvature tensor vanishes on any Riemannian three-manifold. Note that $g_{ij} = h_{ij}$ on $\Sigma^3$. Putting (\ref{Gauss}),(\ref{Rie decomp M}), and (\ref{RIe decomp Sigma}) together, we have
on $\Sigma^3$ that
\begin{align}\label{expansion}
    W_{ikjl} + g_{ij}A_{kl} + g_{kl}A_{ij} - g_{il}A_{jk} - g_{jk}A_{il} = 0,
\end{align}
where 
\begin{align} \label{WA} 
A_{ij} = P_{ij} - P^{\Sigma}_{ij} + \frac{1}{18}H^2g_{ij} = W_{0i0j}.
\end{align}

Next, square both sides of (\ref{expansion}) and combine like terms. To simplify we calculate at $p\in\Sigma$ with respect to Fermi coordinates, so at $p$ we have $$g_{ij} = \delta_{ij},\,\,\, g_{i0} = 0,\,\,\, g_{00} = 1.$$  Also, in the following calculations we adopt the Einstein summation convention.  Since $W$ is trace-free, at $p$ we have 
\begin{align}
    0 = W_{ikil} + W_{0k0l}g_{00}.
\end{align}
Hence by (\ref{WA}) 
\begin{align}
    W_{ikil}A_{kl} = -W_{0k0l}A_{kl} = -|A|^2.
\end{align}
At $p$, 
\begin{align}
    g_{ij}g_{il} = g_{ij}g_{il} + g_{0j}g_{0l} = \delta_{jl},
\end{align}
hence 
\begin{align}
    -g_{ij}A_{kl}g_{il}A_{jk} = -\delta_{jl}A_{kl}A_{jk} = -|A|^2.
\end{align}
Putting everything together, we conclude that 
\begin{align}\label{W_ijkl on the boundary}
    |W_{ijkl}|^2 - 4|A|^2 +4\left(g_{ij}A_{ij}\right)^2 = 0.
\end{align}
Once again using the fact that $W$ is trace-free, 
\begin{align}
    g_{ij}A_{ij} = g_{ij}W_{i0j0} = g_{ij}W_{i0j0} + g_{00}W_{0000} = 0,
\end{align}
hence 
\begin{align}\label{W_{ijkl} on the boundary simplified}
    |W_{ijkl}|^2 = 4|A|^2.
\end{align}
Plugging $A_{ij} = P_{ij} - P^{\Sigma}_{ij} + \frac{1}{18}H^2g_{ij}$ into (\ref{W_{ijkl} on the boundary simplified}), we obtain the desired identity.
\end{proof}

\begin{remark}
There are two model cases for Lemma \ref{Weyl umbilic boundary}:
\begin{itemize}
    \item For the round hemisphere $(S_+^4,S^3,g_{S_+^4})$, we have on $S^3$ that
    \begin{align}
        W = 0,\,\,\,\, P_{ij} = \frac{1}{2}g_{ij},\,\,\,\, P_{ij}^{S^3} = \frac{1}{2}g_{ij},\,\,\,\, H = 0.
    \end{align}
    \item For the flat disc $(B^4,S^3,g_{Eucl})$, we have on $S^3$ that
    \begin{align}
        W = 0,\,\,\,\, P_{ij} = 0,\,\,\,\, P_{ij}^{S^3} = \frac{1}{2}g_{ij},\,\,\,\, H = 3.
    \end{align}
\end{itemize}
\end{remark}

From Lemma \ref{Weyl umbilic boundary}, it is natural to ask under what conditions the Weyl curvature is vanishing on the boundary. The following lemma reveals that the Weyl curvature is vanishing on the boundary under appropriate conformally invariant conditions. This lemma may be of some independent interest.

\begin{lemma}\label{Weyl boundary lemma}
Suppose $(M^4,\Sigma^3,g)$ satisfies
\begin{itemize}
    \item $B_g = 0$ in $M$;
    \item $S_g = 0$ on $\Sigma$;
    \item $(\Sigma^3,h)$ is umbilic;
    \item $(\Sigma^3,h)$ is conformally equivalent to a three-dimensional space form.
\end{itemize}
Then $W \equiv 0$ on $\Sigma^3$.
\end{lemma}

\begin{proof}
Since the conditions and the conclusion are both conformally invariant, we may assume that the boundary $(\Sigma^3,h)$ is isometric to a three-dimensional space form after a conformal transformation of the metric. In any case, we may scale the metric to obtain $P^{\Sigma}_{ij} = \frac{1}{2}ch_{ij}$, where $c=0,\pm1$.

Recall the Bach-flat condition is 
\begin{align*}
    0 = B_{\alpha\beta} = \nabla^{\gamma}\nabla^{\delta} W_{\alpha\gamma\beta\delta} + P^{\gamma\delta}W_{\alpha\gamma\beta\delta}.
\end{align*}

If we consider the pure normal directions of Bach tensor, we have
\begin{align}
    0 = \nabla^{\gamma}\nabla^{\delta} W_{0\gamma0\delta} + P^{\gamma\delta}W_{0\gamma0\delta}.
\end{align}
From the symmetry of Weyl curvature, this implies 
\begin{align} \label{Bach on sigma}
    0 = \nabla^{\gamma}\nabla^{\delta} W_{0\gamma0\delta} + P^{ij}W_{0i0j}.
\end{align}
From Lemma \ref{Weyl umbilic boundary}, we have on $\Sigma^3$ that
\begin{align}
    W_{0i0j} = P_{ij} - P^{\Sigma}_{ij} + \frac{1}{18}H^2g_{ij}
\end{align}
and thereby
\begin{align} \label{P^ij}
    P^{ij} = g^{\alpha{i}}g^{\beta{j}}P_{\alpha\beta} = g^{k{i}}g^{l{j}}P_{kl} = h^{k{i}}h^{l{j}}\left(P^{\Sigma}_{kl} + W_{0k0l} - \frac{1}{18}H^2h_{kl}\right).
\end{align}
Plugging (\ref{P^ij}) into (\ref{Bach on sigma}), we have on $\Sigma^3$
\begin{align}\label{Bach expansion}
    0 = \nabla^{\gamma}\nabla^{\delta} W_{0\gamma0\delta} + W_{0i0j}h^{k{i}}h^{l{j}}\left(P^{\Sigma}_{kl} + W_{0k0l} - \frac{1}{18}H^2h_{kl}\right).
\end{align}
Since $P^{\Sigma}_{kl} = \frac{1}{2}ch_{kl}$, we have
\begin{align}\label{Simplify 1}
    W_{0i0j}h^{k{i}}h^{l{j}}P^{\Sigma}_{kl} = \frac{1}{2}ch^{ij}W_{0i0j} = \frac{1}{2}cg^{\alpha\beta}W_{0\alpha0\beta} = 0,
\end{align}
and
\begin{align}\label{Simplify 2}
    \frac{1}{18}H^2W_{0i0j}h^{k{i}}h^{l{j}}h_{kl} = \frac{1}{18}H^2h^{ij}W_{0i0j} = 0.
\end{align}

Plugingg (\ref{Simplify 1}) and (\ref{Simplify 2}) into (\ref{Bach expansion}), we obtain on $\Sigma^3$
\begin{align}\label{Bach expansion simplified}
    0 = \nabla^{\gamma}\nabla^{\delta} W_{0\gamma0\delta} + |W_{0i0j}|_{\Sigma}^2.
\end{align}

We now simplify the first term in (\ref{Bach expansion simplified}). To simplify, we once again use Fermi coordinates based at a point $p \in \Sigma^3$.  Then at $p$, 
\begin{align} \label{Ch1} \begin{split} 
    & \Gamma_{ij}^k = 0, \,\,\,\,\,\, \Gamma_{ij}^{0} = L_{ij} = \frac{1}{3}Hg_{ij}, \\
    & \Gamma_{i0}^j = - L_{ij} = -\frac{1}{3}Hg_{ij}, \,\,\,\,\,\, \Gamma_{i0}^0 = \Gamma_{00}^0 = 0.
    \end{split}
\end{align}
Then 
\begin{align}
    \nabla_{\gamma}\nabla_{\delta} W_{0\gamma0\delta} = \nabla_{0}\nabla_{0} W_{0000} + \nabla_{i}\nabla_{0} W_{0i00} +
    \nabla_{0}\nabla_{i} W_{000i}+ \nabla_{i}\nabla_{j} W_{0i0j}
\end{align}
Using symmetries of the Weyl tensor, we note that 
\begin{align}\label{00 W_0000}
    \nabla_{0}\nabla_{0} W_{0000} = \nabla_{\delta} W_{\alpha\beta\gamma\gamma} = 0.
\end{align}
We now calculate $\nabla_{i}\nabla_{0} W_{0i00}$ at $p\in\Sigma^3$ using (\ref{Ch1}): 
\begin{align}
\begin{split}
    \nabla_{i}\nabla_{0} W_{0i00} & = \partial_i(\nabla_{0} W_{0i00}) - \Gamma_{i0}^{\alpha}\nabla_{\alpha}W_{0i00} - \Gamma_{i0}^{\alpha}\nabla_{0} W_{\alpha{i00}} - \Gamma_{ii}^{\alpha}\nabla_{0} W_{0\alpha{00}} \\
    & - \Gamma_{i0}^{\alpha}\nabla_{0} W_{0i\alpha{0}} - \Gamma_{i0}^{\alpha}\nabla_{0}W_{0i0\alpha} \\
    & = - \Gamma_{i0}^{j}\nabla_{0} W_{0ij{0}} - \Gamma_{i0}^{j}\nabla_{0}W_{0i0j}  \\
    & = 0
\end{split}
\end{align}
where we have used once again that all contractions of $W$ vanish.  Thus, we have at $p\in\Sigma^3$
\begin{align}\label{i0 W_0i00}
    \nabla_{i}\nabla_{0} W_{0i00} = 0,
\end{align}
and similarly at $p$ 
\begin{align}
    \nabla_{i}\nabla_{0} W_{000i} = 0.
\end{align}

Next, calculate $\nabla_{0}\nabla_{i} W_{000i}$ at $p\in\Sigma^3$. By the Ricci identity and symmetry of curvature tensor, 
\begin{align}\label{0i W_000i}
    \nabla_{0}\nabla_{i} W_{000i} = \nabla_{i}\nabla_{0} W_{000i} - R_{j{00i}}W_{j00i} - R_{j0{0i}}W_{0j0i} = 0.
\end{align}

We now claim that at $p$, 
\begin{align}\label{second order Weyl}
    \nabla_{i}\nabla_{j} W_{0i0j} = \nabla^{\Sigma}_{i}\nabla^{\Sigma}_{j} W_{0i0j}.
\end{align}
To see this, first note 
\begin{align}
    \nabla^{\Sigma}_{k}W_{0i0j} = \partial_kW_{0i0j},
\end{align}
hence 
\begin{align}
\begin{split}
    \nabla_{k}W_{0i0j} & = \partial_kW_{0i0j} - \Gamma_{k0}^{\alpha}W_{\alpha{i0j}} - \Gamma_{ki}^{\alpha}W_{0\alpha{0j}} - \Gamma_{k0}^{\alpha}W_{{0i}\alpha{j}} - \Gamma_{kj}^{\alpha}W_{0i0\alpha} \\
                     & = \partial_kW_{0i0j} - \Gamma_{k0}^{m}W_{m{i0j}} - \Gamma_{ki}^{m}W_{0m{0j}} - \Gamma_{k0}^{m}W_{{0i}m{j}} - \Gamma_{kj}^{m}W_{0i0m} \\
                     & = \partial_kW_{0i0j} - \Gamma_{k0}^{m}W_{m{i0j}} - \Gamma_{k0}^{m}W_{{0i}m{j}} \\
                     & = \nabla^{\Sigma}_{k}W_{0i0j},
\end{split}
\end{align}
where we have used $W_{ijk0} = 0$ on $\Sigma^3$ by Lemma \ref{Weyl umbilic boundary}.  Therefore, 
\begin{align}
    \nabla_{k}W_{0i0j} = \nabla^{\Sigma}_{k}W_{0i0j}.
\end{align}
Also, 
\begin{align}
    \nabla^{\Sigma}_{i}\nabla^{\Sigma}_{j}W_{0i0j} = \partial_i{(\nabla^{\Sigma}_{j}W_{0i0j})}.
\end{align}
It follows that 
\begin{align} \label{HW} 
\begin{split}
    \nabla_{i}\nabla_{j} W_{0i0j} & = \partial_i(\nabla_{j} W_{0i0j}) - \Gamma_{ij}^{\alpha}\nabla_{\alpha}W_{0i0j} - \Gamma_{i0}^{\alpha}\nabla_{j} W_{\alpha{i0j}} - \Gamma_{ii}^{\alpha}\nabla_{j} W_{0\alpha{0j}} \\
    & - \Gamma_{i0}^{\alpha}\nabla_{j} W_{0i\alpha{j}} - \Gamma_{ij}^{\alpha}\nabla_{j}W_{0i0\alpha} \\
                                  & = \partial_i(\nabla^{\Sigma}_{j}W_{0i0j}) - \Gamma_{ij}^{0}\nabla_{0}W_{0i0j} - \Gamma_{i0}^{k}\nabla_{j} W_{k{i0j}} - \Gamma_{i0}^{k}\nabla_{j} W_{0ik{j}} \\
                                  & = \nabla^{\Sigma}_{i}\nabla^{\Sigma}_{j}W_{0i0j} - \Gamma_{ij}^{0}\nabla_{0}W_{0i0j} - \Gamma_{i0}^{k}\nabla_{j} W_{k{i0j}} - \Gamma_{i0}^{k}\nabla_{j} W_{0ik{j}},
\end{split}
\end{align}
where we have used that 
\begin{align}
    \Gamma_{ii}^{\alpha}\nabla_{j} W_{0\alpha{0j}} = \Gamma_{ii}^{0}\nabla_{j} W_{00{0j}} = 0
\end{align}
and
\begin{align}
    \Gamma_{ij}^{\alpha}\nabla_{j}W_{0i0\alpha} = \Gamma_{ij}^{0}\nabla_{j}W_{0i00} = 0.
\end{align}
Note that $W_{ijk0} = 0$ on $\Sigma^3$. Now we calculate the last three terms in (\ref{HW}): 
\begin{align}\label{Gnabla}
    \Gamma_{i0}^{k}\nabla_{j} W_{ki0j} = -\Gamma_{i0}^{k}\Gamma_{j0}^{m} W_{kimj} -\Gamma_{i0}^{k} \Gamma_{jk}^{0} W_{0i0j} -\Gamma_{i0}^{k} \Gamma_{ji}^{0} W_{k00j} -\Gamma_{i0}^{k}\Gamma_{jj}^{0} W_{ki00} = 0
\end{align}
\begin{align}
    \Gamma_{i0}^{k}\nabla_{j} W_{0ikj} = -\Gamma_{i0}^{k}\Gamma_{j0}^{m} W_{mikj} -\Gamma_{i0}^{k}\Gamma_{ji}^{0} W_{00kj} -\Gamma_{i0}^{k}\Gamma_{jk}^{0} W_{0i0j} -\Gamma_{i0}^{k}\Gamma_{jj}^{0} W_{0ik0} = 0
\end{align}

Since $S_g = 0$ on $\Sigma$, 
\begin{align}
    0 = \nabla_{m}W_{mi0j} + \nabla_{k}W_{kj0i} + \nabla_0W_{0i0j} + \frac{4}{3}HW_{0i0j},
\end{align}
which implies at $p$ 
\begin{align}
    \nabla^{0}W_{0i0j} = - \nabla_{m}W_{mi0j} - \nabla_{k}W_{kj0i} - \frac{4}{3}HW_{0i0j}.
\end{align}
Therefore, 
\begin{align}
    \Gamma_{ij}^{0}\nabla^{0}W_{0i0j} = -\frac{1}{3}Hg_{ij}\left(\nabla^{m}W_{mi0j} + \nabla^{k}W_{kj0i} + \frac{4}{3}HW_{0i0j}\right) = 0,
\end{align}
where the last equality follows the same way as (\ref{Gnabla}) is established.  

To summarize, we have shown that 
\begin{align}
    \nabla_{i}\nabla_{j} W_{0i0j} = \nabla^{\Sigma}_{i}\nabla^{\Sigma}_{j} W_{0i0j},
\end{align}
which is (\ref{second order Weyl}).

Plugging (\ref{00 W_0000}), (\ref{i0 W_0i00}), (\ref{0i W_000i}), and (\ref{second order Weyl}) into (\ref{Bach expansion simplified}), we have on $\Sigma^3$
\begin{align}
    0 = \nabla_{\Sigma}^{i}\nabla_{\Sigma}^{j} W_{0i0j} + \big|W_{0i0j}\big|_{\Sigma}^2.
\end{align}
Integrating this over $\Sigma^3$ and using the divergence theorem, we conclude that $W_{0i0j} \equiv 0$ on $\Sigma^3$.
It follows from Lemma \ref{Weyl umbilic boundary} that $W \equiv 0$ on $\Sigma^3$.
\end{proof}


\begin{remark}
If $(X^4, N^3,g_{+})$ is a CCE four-manifold, then any compactification $\bar{g} = \rho^2 g_{+}$ has $W_{\bar{g}} \vert_{N^3} = 0$; see \cite{CG18}, Lemma 2.3.
\end{remark}

\section{Expansion of the metric near the boundary}  \label{gT} 

In this section we compute the expansion of the metric near the boundary that will be used in the proof of Theorem \ref{main}.  Although some of the terms in the expansion are well known, we will need the precise form up to order four.  Also, we carry out the calculations in arbitrary dimension.

Suppose $(M^n,\Sigma^{n-1},g)$ is a smooth manifold with boundary and $g$ is a Riemannian metric smooth up to the boundary.  Let $\{ x^i \}$ be local coordinates on $\Sigma^{n-1}$. If $r$ is the distance function to $\Sigma^{n-1}$, then we can identify a collar neighborhood of the boundary with $\Sigma^{n-1}\times{[0,\epsilon)}$, with coordinates given by $(x_i,r)$.  We want to compute the expansion of $g$ in $\Sigma^{n-1}\times{[0,\epsilon)}$.  In $\Sigma^{n-1}\times[0,\epsilon)$, write the metric $g$ as
\begin{align}
    g = dr^2 + h_{ij}(x,r)dx^idx^j,
\end{align}
where
\begin{align}
    h_{ij} = \left\langle{\partial_i,\partial_j}\right\rangle.
\end{align}
The first derivative is given by
\begin{align}
    \frac{\partial}{\partial{r}}h_{ij} = \left\langle{\nabla_{\partial_{r}}\partial_i,\partial_j}\right\rangle + \left\langle{\nabla_{\partial_{r}}\partial_j,\partial_i}\right\rangle
\end{align}
Note that
\begin{align}\label{2nd ff}
    \nabla_{\partial_{i}}\partial_r = \nabla_{\partial_{r}}\partial_i = -L_{il}h^{lk}{\partial_{k}}.
\end{align}
Hence,
\begin{align}
    \frac{\partial}{\partial{r}}h_{ij} = -2L_{ij}.
\end{align}

The second derivative is given by
\begin{align}
\begin{split}
    \frac{\partial^2}{\partial{r^2}}h_{ij} & = \left\langle{\nabla_{\partial_{r}}\nabla_{\partial_{r}}\partial_i,\partial_j}\right\rangle + 2\left\langle{\nabla_{\partial_{r}}\partial_i,\nabla_{\partial_{r}}\partial_j}\right\rangle + \left\langle{\nabla_{\partial_{r}}\nabla_{\partial_{r}}\partial_j,\partial_i}\right\rangle \\
     & = \left\langle{\nabla_{\partial_{r}}\nabla_{\partial_{r}}\partial_i,\partial_j}\right\rangle + \left\langle{\nabla_{\partial_{r}}\nabla_{\partial_{r}}\partial_j,\partial_i}\right\rangle + 2L_{ik}L_{j}^k
\end{split}
\end{align}
From the Jacobi field equation, we have
\begin{align}\label{Jaboci field}
    \nabla_{\partial_{r}}\nabla_{\partial_{r}}\partial_i = -R_{0i0}^k\partial_{k}
\end{align}
Hence, we have
\begin{align}
    \left\langle{\nabla_{\partial_{r}}\nabla_{\partial_{r}}\partial_i,\partial_j}\right\rangle = - R_{0i0j}
\end{align}
and thereby
\begin{align}
    \frac{\partial^2}{\partial{r^2}}h_{ij} = -2R_{0i0j} + 2L_{ik}L_{j}^k
\end{align}

The third derivative is given by
\begin{align}
\begin{split}
    \frac{\partial^3}{\partial{r^3}}h_{ij} & = \left\langle{\nabla_{\partial_r}\nabla_{\partial_{r}}\nabla_{\partial_{r}}\partial_i,\partial_j}\right\rangle + \left\langle{\nabla_{\partial_{r}}\nabla_{\partial_{r}}\nabla_{\partial_{r}}\partial_j,\partial_i}\right\rangle \\
    & + 3\left\langle{\nabla_{\partial_{r}}\nabla_{\partial_{r}}\partial_i,\nabla_{\partial_{r}}\partial_j}\right\rangle + 3\left\langle{\nabla_{\partial_{r}}\nabla_{\partial_{r}}\partial_j,\nabla_{\partial_{r}}\partial_i}\right\rangle
\end{split}
\end{align}
By (\ref{2nd ff}) and (\ref{Jaboci field}), We calculate
\begin{align}
    \left\langle{\nabla_{\partial_{r}}\nabla_{\partial_{r}}\partial_i,\nabla_{\partial_{r}}\partial_j}\right\rangle = R_{0i0}^kh_{km}L_{jl}h^{lm} = R_{0i0k}L^{k}_{j}
\end{align}


The right hand side of (\ref{Jaboci field}) can be understood as the contraction of two tensors. We may take the covariant derivative:
\begin{align}\label{3rd derivative covariant}
    \nabla_{\partial_{r}}\nabla_{\partial_{r}}\nabla_{\partial_{r}}\partial_i = -\nabla_0R_{0i0}^k\partial_{k} - R_{0i0}^k\nabla_{\partial_{r}}\partial_{k}, 
\end{align}
which implies
\begin{align}
    \left\langle{\nabla_{\partial_r}\nabla_{\partial_{r}}\nabla_{\partial_{r}}\partial_i,\partial_j}\right\rangle = -\nabla_0 R_{0i0j} + L^k_{j}R_{i0k0}.
\end{align}
This identity easily implies
\begin{align}\label{h 3rd derivative}
    \frac{\partial^3}{\partial{r^3}}h_{ij} = -2\nabla_0{R}_{0i0j}+8L^k_{(i}{R}_{j)0k0},
\end{align}
where parentheses around a pair of subscripts denotes symmetrization in that pair. 

The fourth derivative is given by
\begin{align}
\begin{split}
    \frac{\partial^4}{\partial{r^4}}h_{ij} & = \left\langle{\nabla_{\partial_{r}}\nabla_{\partial_r}\nabla_{\partial_{r}}\nabla_{\partial_{r}}\partial_i,\partial_j}\right\rangle + \left\langle{\nabla_{\partial_{r}}\nabla_{\partial_{r}}\nabla_{\partial_{r}}\nabla_{\partial_{r}}\partial_j,\partial_i}\right\rangle \\
    & + 4\left\langle{\nabla_{\partial_{r}}\nabla_{\partial_{r}}\nabla_{\partial_{r}}\partial_i,\nabla_{\partial_{r}}\partial_j}\right\rangle + 4\left\langle{\nabla_{\partial_{r}}\nabla_{\partial_{r}}\nabla_{\partial_{r}}\partial_j,\nabla_{\partial_{r}}\partial_i}\right\rangle \\
    & + 6\left\langle{\nabla_{\partial_{r}}\nabla_{\partial_{r}}\partial_j,\nabla_{\partial_{r}}\nabla_{\partial_{r}}\partial_i}\right\rangle
\end{split}
\end{align}
By (\ref{2nd ff})(\ref{Jaboci field}) and (\ref{3rd derivative covariant}), We calculate
\begin{align}\label{31}
    \left\langle{\nabla_{\partial_{r}}\nabla_{\partial_{r}}\nabla_{\partial_{r}}\partial_i,\nabla_{\partial_{r}}\partial_j}\right\rangle = \nabla_0R_{0i0k}L_{j}^{k} - R_{0i0k}L_{l}^{k}L_{j}^l
\end{align}
\begin{align}\label{22}
    \left\langle{\nabla_{\partial_{r}}\nabla_{\partial_{r}}\partial_j,\nabla_{\partial_{r}}\nabla_{\partial_{r}}\partial_i}\right\rangle = R_{0j0}^k{R}_{0i0k}
\end{align}
For the term $\left\langle{\nabla_{\partial_{r}}\nabla_{\partial_r}\nabla_{\partial_{r}}\nabla_{\partial_{r}}\partial_i,\partial_j}\right\rangle$, we take the covariant derivative of (\ref{3rd derivative covariant}) to obtain
\begin{align}
    \nabla_{\partial_{r}}\nabla_{\partial_{r}}\nabla_{\partial_{r}}\nabla_{\partial_{r}}\partial_i = -\nabla_0\nabla_0R_{0i0}^k\partial_{k} - 2\nabla_0R_{0i0}^k\nabla_{\partial_{r}}\partial_{k} - R_{0i0}^k\nabla_{\partial_{r}}\nabla_{\partial_{r}}\partial_{k},
\end{align}
which implies
\begin{align}\label{40}
    \left\langle{\nabla_{\partial_{r}}\nabla_{\partial_r}\nabla_{\partial_{r}}\nabla_{\partial_{r}}\partial_i,\partial_j}\right\rangle =  -\nabla_0\nabla_0R_{0i0j} + 2\nabla_0R_{0i0k}L_{j}^k + R_{0i0}^k{R}_{0j0k}
\end{align}
Putting (\ref{31})(\ref{22}) and (\ref{40}) together, we have
\begin{align}\label{h 4th derivative}
\begin{split}
     \frac{\partial^4}{\partial{r^4}}h_{ij} = & -2\nabla_0\nabla_0R_{0i0j} + 6\nabla_0{R_{0i0k}}L_{j}^{k} + 6\nabla_0{R_{0j0k}}L_{i}^{k} \\
    & - 4R_{0i0k}L_{l}^{k}L_{j}^l - 4R_{0j0k}L_{l}^{k}L_{i}^l + 8R_{0i0}^k{R}_{0j0k}
\end{split}
\end{align}

We summarize the preceding calculations in the following lemma: 

\begin{lemma}\label{boundary expansion}
Suppose $(M^{n}, \Sigma^{n-1}, g)$ is a Riemannian manifold with boundary. Then we have the expansion for metric $g$ in $\Sigma\times[0,\epsilon)$
\begin{align}
    g = dr^2 + h_{ij}(x,r)dx^idx^j
\end{align}
where
\begin{align}
    h_{ij}(x,r) = h_{ij}^{(0)} + rh_{ij}^{(1)} + \frac{r^2}{2!}h_{ij}^{(2)} + \frac{r^3}{3!}h_{ij}^{(3)} + \frac{r^4}{4!}h_{ij}^{(4)} + O(r^5)
\end{align}
where $h_{ij}^{(k)}$ are symmetric $2$-tensors defined on $\Sigma^{n-1}$
\begin{align} \label{TayCoeff} 
    \begin{split}
        h_{ij}^{(0)} = & \,\,\, g_{ij} \\
        h_{ij}^{(1)} = & -2L_{ij} \\
        h_{ij}^{(2)} = & -2R_{0i0j} + 2L_{ik}L_{j}^k \\
        h_{ij}^{(3)} = & -2\nabla_0{R}_{0i0j}+4L^k_{i}{R}_{j0k0} + 4L^k_{j}{R}_{i0k0}, \\
        h_{ij}^{(4)} = & -2\nabla_0\nabla_0R_{0i0j} + 6\nabla_0{R_{0i0k}}L_{j}^{k} + 6\nabla_0{R_{0j0k}}L_{i}^{k} \\
    & - 4R_{0i0k}L_{l}^{k}L_{j}^l - 4R_{0j0k}L_{l}^{k}L_{i}^l + 8R_{0i0}^k{R}_{0j0k}
    \end{split}
\end{align}
\end{lemma}

\section{The proof of Theorem \ref{main}} 

The proof of Theorem \ref{main} follows the outline of the proof of Theorem A given by Hang-Wang, and can be divided into two steps.  The first step is to show that the metric near the boundary has an expansion coinciding with the round upper hemisphere up to arbitrary order.  The the second step is to use the analyticity of Bach-flat metrics with constant scalar curvature (see below), to show that the manifold has constant sectional curvature.  We remark that in the proof of the Hang-Wang result, since the Einstein condition is second order in the metric they only needed the explicit expansion of the metric up to second order. In our setting, since the Bach-flat condition is fourth order we needed to calculate the expansion of metric to the fourth order in the previous section.  

Assume $(M^4,\Sigma^3,g)$ is Bach-flat with $S$-flat and umbilic boundary.  Since these assumptions are conformally invariant, we may further assume that $g$ is a Yamabe metric with scalar curvature normalized so that $R_g = 12$ and totally geodesic boundary.  Finally, we assume that the induced metric $h = g\vert_{\Sigma^3}$ is isometric to the standard metric on $S^3$.  Under these assumptions, we have on $\Sigma^3$ 
\begin{align} \label{RS} \begin{split} 
P^{\Sigma}_{ij} &= \frac{1}{2}h_{ij},  \\
R^{\Sigma} &= 6. 
\end{split}
\end{align} 
Then the Gauss curvature equations imply (see (\ref{Gauss 2nd trace})) on $\Sigma^3$ that
\begin{align} \label{P00}
    P_{00} = \frac{1}{2}.
\end{align}
Also, from Lemma \ref{Weyl boundary lemma} we conclude 
\begin{align} \label{WS} 
W \big|_{\Sigma^3} \equiv 0. 
\end{align}
The vanishing of the Weyl tensor on the boundary implies, by Lemma \ref{Weyl umbilic boundary}, that the Schouten tensor of $g$ satisfies on $\Sigma^3$
\begin{align} \label{PS} 
    P_{ij} = \dfrac{1}{2}h_{ij}.
\end{align}
Using the decomposition of curvature tensor along with (\ref{WS}), (\ref{P00}), and (\ref{PS}),  we obtain on $\Sigma^3$ that
\begin{align} \label{RZS} 
    R_{0i0j} = W_{0i0j} + P_{00}g_{ij} + P_{ij}g_{00} = h_{ij}.
\end{align}

Recall from Section \ref{gT} that near the boundary, the metric $g$ can be expressed as
\begin{align} \label{gexp} 
    g = dr^2 + h_{ij}(x,r)dx^i dx^j, 
\end{align}
and by Lemma \ref{boundary expansion}, $h_{ij}(x,r)$ has the expansion (up to order four) 
\begin{align}  \label{gTaylor}
h_{ij}(x,r) = h_{ij}^{(0)} + rh_{ij}^{(1)} + \frac{r^2}{2!}h_{ij}^{(2)} + \frac{r^3}{3!}h_{ij}^{(3)} + O(r^4),
\end{align}
and $h_{ij}^{(k)}$ are given by (\ref{TayCoeff}).  In particular, by (\ref{RZS}) and the fact that $\Sigma^3$ is totally geodesic we immediately have 
\begin{align} \label{first2} 
\begin{split}
    h^{(1)}_{ij} & = 0 \\
    h^{(2)}_{ij} & = -2R_{0i0j} = -2g_{ij}.
\end{split}
\end{align}
To determine $h^{(3)}_{ij}$, we need the following result from \cite{GZ20}:

\begin{lemma}\label{S equal to h3}
Suppose $(M^4,\Sigma^3,g)$ is a smooth Riemannian manifold with constant scalar curvature and totally geodesic boundary. Then 
\begin{align}
    h^{(3)}_{ij} = -4S_{ij}.
\end{align}
\end{lemma}

Combining (\ref{first2}) with Lemma \ref{S equal to h3}, we conclude 
\begin{align} \label{h4}
    h_{ij}(x,r) = \cos^2{(r)}h_{ij}(x,0) + O(r^4), \,\,\,\, as \,\,\,\, r \to 0.
\end{align}

\begin{lemma}\label{expansion m}For every integer $m \geq 1$,
\begin{align}\label{inifite order}
    h_{ij}(x,r) = \cos^2{(r)}h_{ij}(x,0) + O(r^m), \,\,\,\, as \,\,\,\, r \to 0
\end{align}
\end{lemma}

\begin{proof}[Proof of Lemma \ref{expansion m}] We have already established this identity for $m=1,2,3,4$.  The proof for general $m$ will follow from induction.  

Suppose (\ref{inifite order}) is valid for some $m \geq 4$. Our strategy is to calculate the fourth order derivative $h^{(4)}_{ij}$ with the Bach-flat condition and get an improvement on the order of derivatives in $r$. Without loss of generality, we may calculate in Fermi coordinates based at $p\in\Sigma^3$ and assume $h_{ij}(p,0) = \delta_{ij}$ and have $h_{ij} = \cos^2(r)\delta_{ij} + O(r^{m})$ by induction hypothesis. Note that $g_{0i} = 0$, $g_{00} = 1$ and 
\begin{align}
    h^{ij} = \frac{1}{\cos^{2}(r)}\delta_{ij} + O(r^{m}).
\end{align}
Note that $L_{ij}$ only involves differentiating the metric in $r$ once. Hence, we have
\begin{align}\label{2nd ff expansion}
    L_{ij} = \sin(r)\cos(r)\delta_{ij} + O(r^{m-1})
\end{align}
Note that curvature tensor only involves differentiating the metric in $r$ twice. Hence, we have
\begin{align}\label{Rm expansion}
    \begin{split}
          & R_{0i0j} = \cos^2(r)\delta_{ij} + O(r^{m-2}) \\
          & R_{ikj0} =  O(r^{m-2}) \\
          & R_{ikjl} = \cos^4(r)(\delta_{ij}\delta_{kl}-\delta_{il}\delta_{kj}) + O(r^{m-2}) \\
          & P_{ij} = \frac{1}{2}\cos^2(r)\delta_{ij} + O(r^{m-2}) \\
          & P_{i0} = O(r^{m-2}) \\
          & P_{00} = \frac{1}{2} + O(r^{m-2}) \\
          & R = 12 + O(r^{m-2}) \\
          & W = O(r^{m-2}) 
    \end{split}
\end{align}
Note that the derivative of curvature tensor only involves differentiating the metric in $r$ three times. Hence, we have
\begin{align}\label{normal Rm expansion}
    \nabla{Rm} = O(r^{m-3}) 
\end{align}

Finally we need to deal with the item $\nabla_0\nabla_0R_{0i0j}$ which involves differentiating the metric in $r$ four times. The Bach-flat condition will enable us to reduce the order of differentiation of the metric in $r$. The decomposition of curvature implies
\begin{align}\label{00R0i0j}
    \nabla_0\nabla_0R_{0i0j} = \nabla_0\nabla_0W_{0i0j} + \nabla_0\nabla_0P_{00}g_{ij} + \nabla_0\nabla_0P_{ij}.
\end{align}
Note that the metric is Bach-flat and has constant scalar curvature. By Bianchi identities, the Bach-flat condition can be written in two ways \cite{Der83}:
\begin{align}\label{Bach two ways}
\begin{split}
    & \nabla^{\gamma}\nabla^{\delta}W_{\alpha\gamma\beta\delta} + P^{\gamma\delta}W_{\alpha\gamma\beta\delta} = 0, \\
    & \Delta{P_{\alpha\beta}} - \frac{1}{6}\nabla_{\alpha}\nabla_{\beta}R + R_{\alpha\gamma\beta\delta}P^{\gamma\delta} - R_{\alpha\gamma}P^{\gamma}_{\beta} + P^{\gamma\delta}W_{\alpha\gamma\beta\delta} = 0.
\end{split}
\end{align}

Note that $\nabla_0\nabla_0W_{0i0j} = \nabla^0\nabla^0W_{0i0j}$. Now we simplify the three items in the right hand side of (\ref{00R0i0j}) by (\ref{Bach two ways}):
\begin{align}
\begin{split}
     \nabla_0\nabla_0W_{0i0j} & = \nabla^{\gamma}\nabla^{\delta}W_{{\gamma}i{\delta}j} - \nabla^{k}\nabla^{l}W_{kilj} - \nabla^{0}\nabla^{l}W_{{0}i{l}j} - \nabla^{k}\nabla^{0}W_{{k}i{0}j} \\
                             & = - P^{\gamma\delta}W_{{i}\gamma{j}\delta} - \nabla^{k}\nabla^{l}W_{kilj} - \nabla^{0}\nabla^{l}W_{{0}i{l}j} - \nabla^{k}\nabla^{0}W_{{k}i{0}j}
\end{split}
\end{align}
\begin{align}
\begin{split}
    \nabla_0\nabla_0P_{00} & = \Delta{P_{00}} - \nabla_k\nabla_kP_{00} \\
                           & = - R_{0\gamma0\delta}P^{\gamma\delta} + R_{0\gamma}P^{\gamma}_{0} - P^{\gamma\delta}W_{0\gamma0\delta} - \nabla_k\nabla_kP_{00}
\end{split}
\end{align}
\begin{align}
\begin{split}
     \nabla_0\nabla_0P_{ij} & = \Delta{P_{ij}} - \nabla_k\nabla_kP_{ij} \\
                           & = - R_{{i}\gamma{j}\delta}P^{\gamma\delta} + R_{{i}\gamma}P^{\gamma}_{j} - P^{\gamma\delta}W_{i\gamma{j}\delta} - \nabla_k\nabla_kP_{ij}
\end{split}
\end{align}
Note that the right hand sides only involve differentiating the metric in $r$ three times. By (\ref{Rm expansion}) we have
\begin{align}\label{simplified expansion}
\begin{split}
    \nabla_0\nabla_0W_{0i0j} & = O(r^{m-3}) \\
    \nabla_0\nabla_0P_{00} & = - \frac{1}{\cos^4(r)}\cos^2(r)\cdot\frac{3}{2}\cos^2(r) + \frac{1}{\cos^2(r)}3\cos^2(r)\cdot\frac{1}{2} + O(r^{m-3})\\
    & = O(r^{m-3}) \\
    \nabla_0\nabla_0P_{ij} & = -\frac{1}{2}\cos^2(r)\delta_{ij} - \frac{1}{\cos^4(r)}\cos^4(r)(\delta_{ij}\delta_{kl}-\delta_{il}\delta_{kj})\cdot{\frac{1}{2}\cos^2(r)\delta_{kl}} \\
    & + \frac{1}{\cos^2(r)}3\cos^2(r)\delta_{ik}\cdot\frac{1}{2}\cos^2(r)\delta_{kj} + O(r^{m-3}) = O(r^{m-3})
\end{split}
\end{align}
By (\ref{00R0i0j}) and (\ref{simplified expansion}) we have 
\begin{align}\label{00R0i0j simplified}
    \nabla_0\nabla_0R_{0i0j} = O(r^{m-3}).
\end{align}
Therefore, combining (\ref{Rm expansion})(\ref{normal Rm expansion})(\ref{00R0i0j simplified}) we calculate by (\ref{h 4th derivative})
\begin{align}
    \begin{split}
        \frac{\partial^4}{\partial{r^4}}h_{ij} & = -\dfrac{4\cos^2(r)\sin^2(r)\cos^2(r)}{\cos^4(r)}\delta_{ij} -\dfrac{4\cos^2(r)\sin^2(r)\cos^2(r)}{\cos^4(r)}\delta_{ij} \\
                                               & + 8\dfrac{\cos^4(r)}{\cos^2(r)}\delta_{ij} + O(r^{m-3}) \\
                                               & = 8\cos(2r)\delta_{ij} + O(r^{m-3}).
    \end{split}
\end{align}
This clearly implies that $h_{ij}(x,r) = \cos^2{(r)}h_{ij}(x,0) + O(r^{m+1})$. Hence, the lemma follows from induction.

\end{proof}


Now we consider the double manifold of $(M^4,\Sigma^3,g)$ which is denoted by $(\overline{M},g_d)$. It is easy to see that $(\overline{M},g_d)$ is Bach-flat and has constant scalar curvature. Recall that in harmonic coordinates a Bach-flat metric with constant scalar curvature satisfies an elliptic system of fourth order \cite{TV05a}\cite{TV05b}. Indeed, if the scalar curvature is constant ($R=c$), then we have
\begin{align}
B_{\alpha\beta} = -\frac{1}{2}\Delta{E_{\alpha\beta}} - E^{\gamma\delta}W_{\alpha\gamma\beta\delta} +E_{\alpha}^{\gamma}E_{\beta\gamma} -\frac{1}{4}|E|^2g_{\alpha\beta} + \frac{1}{6}cE_{\alpha\beta}.
\end{align}
By the formula of Ricci tensor in harmonic coordinates in \cite{DK81}, we can write the Bach-flat equation in harmonic coordinates as
\begin{align}
0 =  B_{\alpha\beta}=  \frac{1}{4}g^{\gamma\delta}g^{\mu\lambda}\frac{\partial^4g_{\alpha\beta}}{\partial{x^{\gamma}}\partial{x^{\delta}}\partial{x^{\mu}}\partial{x^{\lambda}}} + \cdots
\end{align}
where the dots indicate terms involving at most three derivatives of the metric and the principal part of Bach tensor is just one quarter of the square of Laplacian. Hence, the metric $g_d$ is real analytic in harmonic coordinates. We define $\Omega$ to be the set of points where $g$ has constant sectional curvature $1$ in a neighborhood. Note that $\Omega$ is nonempty since $\Sigma^3\times(-\epsilon,\epsilon) \subset \Omega$ by (\ref{inifite order}) and the analytic property of metric $g$. Also note that $\Omega$ is an open set by definition. We now show that $\Omega$ conincides with $\overline{M}$. We argue by contradiction. Suppose there is a point $p\in\partial{\Omega}$ satisfying $p\not\in\Omega$. Choose a local harmonic coordinates $y^1,y^2,y^3,y^4$ on a connected neighborhood $U$ of $p$. The analytic functions $R_{ikjl} - g_{ij}g_{kl} + g_{il}g_{jk}$ vanishes on $U\cap\Omega\ne\emptyset$ and thereby vanish identically on $U$. Then $p\in\Omega$, which is a contradiction. Therefore, $\Omega = \overline{M}$ and thereby $(\overline{M},g_d)$ has constant sectional curvature $1$. It is then easy to see that $(M^4,\Sigma^3,g)$ is isometric to $(S^4_+,S^3,g_{S^4_+})$.

\section{The proof of Theorem \ref{SYthm}}

Let $(X^n, g_X)$ and $(Y^n, g_Y)$ be closed, locally conformally flat manifolds with positive scalar curvature.  By Corollary 5 of \cite{SchoenYau}, the connected sum $Z = X \, \sharp \, Y$ obtained by deleting balls around $p$ and $q$ and identifying their boundaries, admits a locally conformally flat metric $\widetilde{g}$ with positive scalar curvature.  This follows from the general surgery result of \cite{SchoenYau}, since the metric $\widetilde{g}$ constructed in Theorem 3 of \cite{SchoenYau} is locally conformally flat in a neighborhood $U$ of the gluing point, and conformal to $g_X$ and $g_Y$ outside of $U$.

Let us apply this result when $X^n = S^4$, $g_X = g_0$ the round metric, $Y^n = S^3 \times S^1$, and $g_Y = h_0 \times d\theta^2$ is the standard product metric. Let $p = (0,0,0,0,1) \in S^4$ be the `north pole' of $S^4 \subset \mathbb{R}^5$, and let $q \in S^3 \times S^1$ be any point.  By the Schoen-Yau construction, there is a locally conformally flat metric $\widetilde{g}$ on $S^4 \, \sharp \, (S^3 \times S^1) \approx S^3 \times S^1$ with positive scalar curvature.  Moreover, $\widetilde{g}$ is conformal to $g_0$ on $S^4 \setminus U$, where $U$ is a small neighborhood of $p$.  In particular, the induced by $\widetilde{g}$ on the equatorial $S^3 = \{ (x^1, \dots , x^5) \in S^4 \, : \, x^5 = 0\} \subset S^4 \, \sharp \, (S^3 \times S^1)$ is umbilic.  Therefore, if $S^4_{-} = \{ (x^1, \dots , x^5) \in S^4 \, : \, x^5 \leq 0\} \approx B^4$ denotes the `lower hemisphere', then $(S^3 \times S^1 \setminus S^4_{-}, \widetilde{g})$ satisfies the conditions of Theorem \ref{SYthm}.

\section{appendix} 

In this appendix we give the proof of Theorem \ref{AppThm}.  Since most of the formulas are fairly standard we will only provide a sketch. 

Let $(M^4, \Sigma^3 = \partial M^4, g)$ be a compact Riemannian manifold with boundary.   Given a symmetric $2$-tensor $v$, let $g(t) = g + tv$; then $g(0) = g$ and $g'(0) = v$.  By the formula for the variation of the metric tensor and volume form, it is readily calculated 
\begin{align*}
\frac{d}{dt} \mathcal{W}(g(t)) \big|_{t=0} &= \frac{d}{dt} \int_{M^4} \| W_{g(t)} \|^2 \, dv_{g(t)}  \big|_{t = 0} \\
&= \frac{d}{dt} \frac{1}{4} \int_{M^4}   (W_{g(t)})^{\alpha \mu \beta \nu} (W_{g(t)})_{\alpha \mu \beta \nu} \, dv_{g(t)}  \big|_{t = 0} \\
&= \int_{M^4} \big\{-   W^{\alpha \mu \beta \nu} \nabla_{\alpha} \nabla_{\beta}  v_{\mu \nu} -  P_{\alpha \beta} W^{\alpha \mu \beta \nu} v_{\mu \nu} \big\} \, dv_g.
\end{align*}
If we integrate by parts in the first term, we obtain 
\begin{align}  \label{D1} 
\frac{d}{dt} \mathcal{W}(g(t)) \big|_{t=0} = -\int_{M^4} B^{\mu \nu} v_{\mu \nu} \, dv_g + \oint_{\Sigma^3} \big\{\nabla_{\alpha} W^{\alpha \mu 0 \nu} v_{\mu \nu} - W^{0 \mu \beta \nu} \nabla_{\beta} v_{\mu \nu} \big\} \, d\sigma_h,
\end{align}
where $B$ is the Bach tensor\footnote{There is often a disagreement in the literature whether the Bach tensor is the $L^2$-gradient of $\mathcal{W}$ or {\em minus} the gradient.}.  Using the convention that Latin indices indicate tangential components, we can rewrite the boundary integrand as
\begin{align} \label{bdyzed} \begin{split}
\oint_{\Sigma^3} & \big\{ \nabla_{\alpha} W^{\alpha \mu 0 \nu} v_{\mu \nu} - W^{0 \mu \beta \nu} \nabla_{\beta} v_{\mu \nu} \big\} \, d\sigma_h \\
&= \oint_{\Sigma^3} \big\{ \nabla_{\alpha}  W^{\alpha 0 0 j} v_{0 j} + \nabla_{\alpha} W^{\alpha i 0 j} v_{i j} - W^{i 0 j 0} \nabla_0 v_{ij} + W^{i 0 j 0} \nabla_j v_{i0} + W^{i 0 j k} \nabla_j v_{ik} \big\} \, d\sigma_h.
\end{split}
\end{align}

Since we are restricting to metrics for which the boundary is umbilic we assume that $v$ preserves the umbilic condition to first order; i.e., 
\begin{align*}
\frac{d}{dt} \big\{ L(g(t))_{ij} - \frac{1}{3} H(g(t)) g(t)_{ij} \big\} \Big|_{t=0} = 0. 
\end{align*}
By standard formulas for the variation of the second fundamental form (see, for example, (2.6) in \cite{Araujo}) this implies 
\begin{align} \label{Lp}
0 = \frac{1}{2} \big( \nabla_i v_{j0} + \nabla_j v_{i0} - \nabla_0 v_{ij} \big) - \frac{1}{3} H v_{ij} - \frac{1}{3} \Big( \big[ \nabla^{\alpha}v_{\alpha 0} - \frac{1}{2} \nabla_0 ( \mbox{tr }v) - \frac{1}{2} \nabla_0 v_{00} \big] - L^{k\ell} v_{k \ell} \Big) h_{ij}.
\end{align} 

\smallskip

\begin{remark}  Symmetric $2$-tensors on $M^4$ whose restriction to the boundary have vanishing trace and satisfy (\ref{Lp}) can be viewed as the formal tangent space to $\mathcal{M}_0(M^4,\Sigma^3)$, the space of Riemannian metrics on $M^4$ with umbilic boundary.  
\end{remark}

\smallskip

Pairing both sides of (\ref{Lp}) with $W^{i0j0}$ and integrating over $\Sigma^3$ gives 
\begin{align} \label{LZ}  
0 = \oint_{\Sigma^3} \big\{ W^{i 0 j 0} \nabla_j v_{i0} - \frac{1}{2} W^{i 0 j 0} \nabla_0 v_{ij} - \frac{1}{3} H W^{i 0 j 0} v_{ij} \big\} \, d\sigma_h,
\end{align}
which we rewrite as
\begin{align} \label{LZ2} 
\oint_{\Sigma^3} W^{i 0 j 0} \nabla_0 v_{ij} \, d\sigma_h = \oint_{\Sigma^3} \big\{ 2 W^{i 0 j 0} \nabla_j v_{i0} - \frac{2}{3} H W^{i 0 j 0} v_{ij} \big\} \, d\sigma_h. 
\end{align}
Substituting this into (\ref{bdyzed}) we obtain 
\begin{align} \label{bdy1} \begin{split}
\oint_{\Sigma^3} & \big\{ \nabla_{\alpha} W^{\alpha \mu 0 \nu} v_{\mu \nu} - W^{0 \mu \beta \nu} \nabla_{\beta} v_{\mu \nu} \big\} \, d\sigma_h \\
&= \oint_{\Sigma^3} \big\{ \nabla_{\alpha}  W^{\alpha 0 0 j} v_{0 j} + \nabla_{\alpha} W^{\alpha i 0 j} v_{i j}   +  \frac{2}{3} H W^{i 0 j 0} v_{ij}   -  W^{i 0 j 0} \nabla_j v_{i0} + W^{i 0 j k} \nabla_j v_{ik} \big\} \, d\sigma_h.
\end{split}
\end{align}

Next, we rewrite the last two terms above via integration by parts.  On $\Sigma^3$, define the symmetric two-tensor $D_{ij} = W_{i 0 j 0}$ and the one-form $\eta_k = v_k$.  More precisely, for tangent vectors $X,Y \in T\Sigma^3$,
\begin{align*}
D(X,Y) &= W(X,\frac{\partial}{\partial r}, Y, \frac{\partial}{\partial r} ), \\
\eta(X) &= v(X,\frac{\partial}{\partial r}). 
\end{align*}
Using the formulas for the Christoffel symbols in (\ref{Ch1}), we can write 
\begin{align} \label{DMS} \begin{split}
\nabla_j v_{i 0} &= \nabla^{\Sigma}_j \eta_i - \frac{1}{3} H h_{ij} v_{00} + \frac{1}{3} H v_{ij}, \\
\nabla^{\Sigma}_j D^{ij} &= \nabla_j W^{i 0 j 0} = \nabla_{\alpha} W^{\alpha 0 i 0}.
\end{split} 
\end{align}
Therefore, we can express the next to last term in (\ref{bdy1}) as 
\begin{align} \label{bdy2}   \begin{split}
\oint_{\Sigma^3} W^{i 0 j 0} \nabla_j v_{i0} \, d\sigma_h &= \oint_{\Sigma^3} W^{i 0 j 0} \{ \nabla^{\Sigma}_j \eta_i  - \frac{1}{3} H h_{ij} v_{00} + \frac{1}{3} H v_{ij} \} \, d\sigma_h  \\
&= \oint_{\Sigma^3} D^{ij}   \nabla^{\Sigma}_j \eta_i  \, d\sigma_h + \frac{1}{3} \oint_{\Sigma^3} H W^{i 0 j 0} v_{ij} \, d\sigma_h.  \\
\end{split}
\end{align}
Integrating by parts on $\Sigma^3$ and using (\ref{DMS}) gives 
\begin{align} \label{bdy3} \begin{split} 
\oint_{\Sigma^3} W^{i 0 j 0} \nabla_j v_{i0} \, d\sigma_h &= - \oint_{\Sigma^3} \nabla^{\Sigma}_j D^{ij} \eta_i  \, d\sigma_h + \frac{1}{3} \oint_{\Sigma^3} H W^{i 0 j 0} v_{ij} \, d\sigma_h \\
&= - \oint_{\Sigma^3} \nabla_{\alpha} W^{\alpha 0 i 0} \eta_i  \, d\sigma_h + \frac{1}{3} \oint_{\Sigma^3} H W^{i 0 j 0} v_{ij} \, d\sigma_h \\
&= - \oint_{\Sigma^3} \nabla_{\alpha} W^{\alpha 0 i 0} v_{i0}  \, d\sigma_h + \frac{1}{3} \oint_{\Sigma^3} H W^{i 0 j 0} v_{ij} \, d\sigma_h \\
&= \oint_{\Sigma^3} \big\{  \nabla_{\alpha} W^{\alpha 0 0 j} v_{j0} + \frac{1}{3} H W^{i 0 j 0} v_{ij} \big\} \, d\sigma_h,
\end{split}
\end{align}
where in the last line we used the symmetries of the Weyl tensor and re-indexed.  

We use a similar argument to rewrite the last term on the right in (\ref{bdy1}).  This time we define the tensor $A$ on the boundary by $A_{ijk} = W_{i 0 jk}$.  Again using the formulas for the Christoffel symbols on $\Sigma^3$, it follows that
\begin{align} \label{DMS2} \begin{split}
\nabla_j v_{ik} &= \nabla^{\Sigma}_j v_{ik} - \frac{1}{3} H v_{0k} h_{ij} - \frac{1}{3} H v_{i 0} h_{jk}, \\
\nabla^{\Sigma}_j A^{ijk} &= \nabla_j W^{i 0 jk} + H W^{i 00k} = \nabla_{\alpha} W^{i 0  \alpha k} - \nabla_0 W^{i 0 0 k} + H W^{i 00k}. 
\end{split}
\end{align}
Therefore, we can express the last term in (\ref{bdy1}) as 
\begin{align} \label{bdy4}  \begin{split} 
\oint_{\Sigma^3} W^{i 0 j k} \nabla_j v_{ik}  \, d\sigma_h &= \oint_{\Sigma^3} W^{i 0 j k} \big\{ \nabla^{\Sigma}_j v_{ik} - \frac{1}{3} H v_{0k} h_{ij} - \frac{1}{3} H v_{i 0} h_{jk} \big\} \, d\sigma_h \\
&= \oint_{\Sigma^3} W^{i 0 j k} \nabla^{\Sigma}_j v_{ik} \, d\sigma_h \\
&= \oint_{\Sigma^3} A^{ijk} \nabla^{\Sigma}_j v_{ik} \, d\sigma_h.
\end{split}
\end{align}
Integrating by parts and using (\ref{DMS2}), we have 
\begin{align} \label{bdy5} \begin{split} 
\oint_{\Sigma^3} W^{i 0 j k} \nabla_j v_{ik}  \, d\sigma_h &= -\oint_{\Sigma^3} \nabla^{\Sigma}_j A^{ijk} v_{ik} \, d\sigma_h \\
&= \oint_{\Sigma^3} \big\{  - \nabla_{\alpha} W^{i 0  \alpha k} v_{ik} + \nabla_0 W^{i 0 0 k} v_{ik} - H W^{i 00k}  v_{ik}  \big\} \, d\sigma_h \\
&= \oint_{\Sigma^3} \big\{  \nabla_{\alpha} W^{\alpha i 0 j} v_{ij} - \nabla_0 W^{i 0 j 0} v_{ij} + H W^{i 0 j 0}  v_{ij}  \big\} \, d\sigma_h,
\end{split} 
\end{align}
where once again we re-indexed and used the symmetries of $W$. 

We now substitute (\ref{bdy3}) and (\ref{bdy5}) into (\ref{bdy1}) to get 
\begin{align} \label{bdy6} \begin{split}
\oint_{\Sigma^3} & \big\{ \nabla_{\alpha} W^{\alpha \mu 0 \nu} v_{\mu \nu} - W^{0 \mu \beta \nu} \nabla_{\beta} v_{\mu \nu} \big\} \, d\sigma_h \\ 
&= \oint_{\Sigma^3} \Big\{  \nabla_{\alpha}  W^{\alpha 0 0 j} v_{0 j} + \nabla_{\alpha} W^{\alpha i 0 j} v_{i j}   +  \frac{2}{3} H W^{i 0 j 0} v_{ij}  - \big[  \nabla_{\alpha} W^{\alpha 0 0 j} v_{j0} - \frac{1}{3} H W^{i 0 j 0} v_{ij} \big] \\
 & \ \ \ + \big[ \nabla_{\alpha} W^{\alpha i 0 j} v_{ij} - \nabla_0 W^{i 0 j 0} v_{ij} + H W^{i 0 j 0}  v_{ij}  \big]  \Big\} \, d\sigma_h \\
 &= \oint_{\Sigma^3} \big\{ 2 \nabla_{\alpha}W^{{\alpha}i0j} v_{ij}  - \nabla_0W^{0i0j} v_{ij} + \frac{4}{3}HW^{0i0j} v_{ij} \big\} \, d\sigma_h.
\end{split}
\end{align}

By the definition of the tensor $S$, 
\begin{align*}
\oint_{\Sigma^3} S^{ij} v_{ij} \, d\sigma_h &= \oint_{\Sigma^3} \big\{ 2 \nabla_{\alpha}W^{{\alpha}i0j} v_{ij}  - \nabla_0W^{0i0j} v_{ij} + \frac{4}{3}HW^{0i0j} v_{ij} \big\} \, d\sigma_h.
\end{align*}
Therefore, from (\ref{bdy6}), (\ref{D1}) and (\ref{bdyzed}) we conclude 
\begin{align}  \label{D2}
\frac{d}{dt} \mathcal{W}(g(t)) \big|_{t=0} = -\int_{M^4} B^{\mu \nu} v_{\mu \nu} \, dv_g + \oint_{\Sigma^3}  S^{ij} v_{ij} \, d\sigma_h.
\end{align}
By restricting to variations supported in the interior of $X^4$, we immediately see that a metric that a critical metric for $\mathcal{W}$ over variations that preserve the umbilic condition to first order must be Bach-flat.  In particular, for any such variation $v$ we have 
 \begin{align}  \label{D3}
\frac{d}{dt} \mathcal{W}(g(t)) \big|_{t=0} = \oint_{\Sigma^3}  S^{ij} v_{ij} \, d\sigma_h.
\end{align}
 To see that $g$ is also $S$-flat, let $v^{\Sigma}$ be a symmetric $2$-tensor defined on $\Sigma^3$, and assume $v^{\Sigma}$ is trace-free with respect to $h = g|_{\Sigma^3}$, the induced metric.  We can extend $v^{\Sigma}$ trivially to a collar neighborhood $U$ of the boundary by using the identification of $U$ with $\Sigma^3 \times [0,\epsilon)$ for some $\epsilon > 0$ small, as described in Section \ref{gT}.   More precisely, if $X,Y$ are tangent vectors defined at a point $p = (x,r) \in U$, then we can write
\begin{align*}
X = X^T + X^0 \frac{\partial}{\partial r}, \ \ Y = Y^T + Y^0 \frac{\partial}{\partial r},
\end{align*}
where $X^T, Y^T \in T_x \Sigma^3$.  Then define $v$ by the formula 
\begin{align*} 
v(X,Y) = v^{\Sigma}(X^T,Y^T).
\end{align*}
We then use a cut-off function to extend $v$ to all of $M^4$, so that $v \equiv 0$ away from $\Sigma^3$.  Note that with respect to the coordinate system in $U$ described in Section \ref{gT}, near $\Sigma^3$ we have
\begin{align} \label{vform}
v = v^{\Sigma}_{ij} \, dx^i \, dx^j. 
\end{align}
By construction it follows that $v$ is trace-free (with respect to $g$), and on $\Sigma^3$ we have  
\begin{align*}
\nabla_0 v_{ij} &= 0, \\
\nabla_0 v_{00} &= 0, \\
\nabla_i v_{j0} &= 0, \\
L^{ij} v_{ij} &= \frac{1}{3} H \, \mbox{tr }v^{\Sigma} = 0. 
\end{align*}
Therefore, $v$ satisfies the constraint in (\ref{Lp}), meaning that variations $g_t$ of the metric $g$ with $\frac{d}{dt}g_t \vert_{t=0} = v$ preserve the umbilic condition to first order.  As we observed above, this implies 
 \begin{align}  \label{D4} \begin{split} 
0 &= \frac{d}{dt} \mathcal{W}(g_t) \big|_{t=0} \\
&= \oint_{\Sigma^3}  S^{ij} v_{ij} \, d\sigma_h \\
& = \oint_{\Sigma^3}  S^{ij} v^{\Sigma}_{ij} \, d\sigma_h. 
\end{split}
\end{align}
Since $v^{\Sigma}$ was an arbitrary trace-free $2$-tensor on $\Sigma^3$, it follows that $S = 0$.

\bibliographystyle{amsplain}

\end{document}